\documentclass{amsart}
\usepackage{amscd,amssymb}
\usepackage{fullpage}
\usepackage{pb-diagram,lamsarrow,pb-lams}

\newcommand{\Aut}       {\operatorname{Aut}}
\newcommand{\Map}       {\operatorname{Map}}
\newcommand{\Perm}      {\operatorname{Perm}}

\newcommand{\dv}        {\operatorname{div}}
\newcommand{\spec}      {\operatorname{spec}}

\newcommand{\aff}       {\mathbb{A}}
\newcommand{\proj}      {\mathbb{P}}

\newcommand{\F}         {{\mathbb{F}}}
              
\newcommand{\Z}         {{\mathbb{Z}}}

\newcommand{\al}        {\alpha}
\newcommand{\bt}        {\beta} 
\newcommand{\gm}        {\gamma}
\newcommand{\lm}        {\lambda}
\newcommand{\lmb}       {\overline{\lambda}}
\newcommand{\sg}        {\sigma}
\newcommand{\om}        {\omega}
\newcommand{\ob}        {{\overline{\omega}}}

\newcommand{\Gm}        {\Gamma}
\newcommand{\Dl}        {\Delta}

\newcommand{\Om}        {\Omega}

\newcommand{\mtx}       {{\mu^\times_3}}
\newcommand{\hE}        {\widehat{E}}
\newcommand{\sm}        {\setminus}
\newcommand{\third}     {{\scriptstyle \frac{1}{3}}}
\newcommand{\tm}        {\times}
\newcommand{\st}        {\;|\;}
\newcommand{\xra}       {\xrightarrow}
\newcommand{\mxi}       {\mathfrak{m}}

\newcommand{\bsm}       {\left(\begin{smallmatrix}}
\newcommand{\esm}       {\end{smallmatrix}\right)}

\newcommand{\CO}       {\mathcal{O}}

\renewcommand{\:}{\colon}

\newtheorem{theorem}{Theorem}

\newtheorem{lemma}[theorem]{Lemma}
\newtheorem{proposition}[theorem]{Proposition}
\newtheorem{corollary}[theorem]{Corollary}
\theoremstyle{definition}
\newtheorem{remark}[theorem]{Remark}

\newenvironment{diag}{
 \renewcommand{\typeout}[1]{}
 \begin{displaymath}
 \begin{diagram}}{
 \end{diagram}
 \end{displaymath}} 

\begin{document}
\title{Level three structures}
\author{N.~P.~Strickland}
\date{\today}
\bibliographystyle{abbrv}

\maketitle 

\section{Introduction}

The theory of elliptic curves and their level structures is important
in stable homotopy theory.  In this note we work out the details of a
certain fragment of the theory where it is possible to be very
explicit.  This is intended as a convenient reference for people
working on elliptic cohomology.  Many of the facts are doubtless
familiar to algebraic geometers.

We will define a scheme $S$ over $\spec(\Z[\third])$, an elliptic curve
$C$ over $S$, and an injective homomorphism $\phi$ from $\F_3^2$ to
the group of sections of $C$.  We will then study the automorphisms of
$S$ and the automorphisms of $C$ that cover them.  The conclusion will
be that $G:=GL_2(\F_3)$ acts on $C$ and $S$ in a way that is compatible
with its evident action on $\F_3^2$.

Next, we will show that $C/S$ is the universal example of an elliptic
curve over a $\Z[\third]$-scheme equipped with a level three
structure.  This could be used to give an alternative construction of
the action of $G$.

We will observe that $C$ is Landweber-exact, and deduce that there is an
elliptic spectrum $E$ attached to $C$, with a compatible action of
$G$.  The spectrum $\hE=L_{K(2)}E$ (at the prime $2$) is a version of
$E_2$; it need not be multiplicatively isomorphic to the more usual
$p$-typical version until we make some algebraic extensions, but I
do not think that that is important.  The whole Hopkins-Miller-Goerss
technology should construct a model of $\hE$ with a rigid action of
$G$, and $EO_2=\hE^{hG}$.

\section{Definition of the curve $C$}

In this note, all schemes are implicitly assumed to be schemes over
$\spec(\Z[\third])$.  We write $\aff^1$ for the affine line, $\mu_3$
for the scheme of cube roots of unity, and $\mtx$ for the subscheme of
primitive cube roots.  We also put 
\begin{align*}
 S_0 &= \aff^1\sm\mu_3 = \proj^1\sm(\mu_3\cup\{\infty\}) \\
 S   &= \mu_3\tm S_0.
\end{align*}
The corresponding rings are
\begin{align*}
 \CO_{\aff^1} &= \Z[\third,\nu] \\
 \CO_{\mu_3}  &= \Z[\third,\om]/(\om^3-1) \\
 \CO_{\mtx}   &= \Z[\third,\om]/(1+\om+\om^2)\\
 \CO_{S_0}    &= \Z[\third,\nu,(\nu^3-1)^{-1}] \\
 \CO_S        &= \Z[\third,\om,\nu,(\nu^3-1)^{-1}]/(1+\om+\om^2).
\end{align*}
We will also use the notation $A=\CO_\mtx$, $B_0=\CO_{S_0}$ and
$B=\CO_S$, and put $\ob=1/\om=\om^2=-1-\om\in A^\tm$.  Note that in $B$
we have $\nu^3-1=(\nu-1)(\nu-\om)(\nu-\ob)$, so that $\nu-1$,
$\nu-\om$ and $\nu-\ob$ are invertible in $B$.

We next define a plane projective cubic curve $C_0$ over $S_0$ by the
homogeneous equation
\[ y^2 z + (\nu^3 - 1) y z^2 + 3 \nu x y z - x^3 = 0. \]
The intersection with $\aff^2\subset\proj^2$ is given by the
inhomogeneous equation
\[ y^2 + (\nu^3 - 1) y + 3 \nu x y - x^3 = 0. \]
We also define $C=C_0\tm_{S_0}S$.

The standard invariants of the plane curve $C$ (with notation as in
Deligne's Formulaire~\cite{de:cef}) are as follows.  Firstly, the
defining equation can be written in the form
\[ y^2 + a_1 x y + a_3 y = x^3 + a_2 x^2 + a_4 x + a_6, \]
where
\begin{align*}
 a_2 &= a_4 = a_6 = 0   \\
 a_1 &= 3\nu            \\
 a_3 &= \nu^3-1.
\end{align*}
The following quantities are defined in terms of the $a_k$ as in the
Formulaire. 
\begin{align*}
 c_4 &= 9\nu(\nu^3+8)                         \\
 c_6 &= 27(\nu^6-20\nu^3-8)                   \\
 \Dl &= 27(\nu^3-1)^3                         \\
 j   &= 27\nu^3(\nu^3+8)^3/(\nu^3-1)^3 
\end{align*}
In particular, this shows that $\Dl$ is a unit in $\CO_S$ so $C$ is an
elliptic curve.

\section{Automorphisms of $S$}

Write $\Om=\{1,\om,\ob,\infty\}$, and let $\Perm(\Om)$ denote the
group of permutations of this set.
\begin{proposition}\label{prop-Aut-S}
 There is a natural isomorphism $\Aut(S)\xra{}\Perm(\Om)$.
\end{proposition}
\begin{proof}
 It is well-known that the ring $\Z[\om]/(1+\om+\om^2)$ is a principal
 ideal domain (because the definition $|a+b\om|=\sqrt{a^2-ab+b^2}$
 gives a Euclidean valuation), and that its group of units is cyclic
 of order $6$, generated by $-\om$.  It follows by standard arguments
 that $A$, $A[\nu]$ and $B$ are unique factorisation domains, and that
 $B^\tm/A^\tm$ is freely generated by $\{\nu-1,\nu-\om,\nu-\ob\}$. 

 Now let $V$ be the set of discrete valuations on $B$ that are trivial
 on $A$, in other words the surjective homomorphisms
 $v\:B^\tm\xra{}\Z$ such that
 \begin{itemize}
  \item[(a)] $v(f)=0$ for $f\in A$
  \item[(b)] $v(f+g)\geq\min(v(f),v(g))$ whenever $f,g,f+g\in B^\tm$. 
 \end{itemize}
 An arbitrary element $f\in B^\tm$ can be written uniquely in the form
 \[ f = a (\nu-1)^{n_1} (\nu-\om)^{n_\om} (\nu-\ob)^{n_\ob} \]
 with $a\in A^\tm$.  We define maps $v_\al\:B^\tm\xra{}\Z$ for
 $\al\in\Om$ by
 \[ v_\al(f) = \begin{cases}
     n_\al & \text{ if } \al\in\{1,\om,\ob\} \\
     -\deg(f)=-n_1-n_\om-n_\ob & \text{ if } \al=\infty.
    \end{cases}
 \]
 One can check directly that these maps are elements of $V$; we next
 claim that there are no more elements.  To see this, suppose that
 $w\in V$, and put $m_\al=w(\nu-\al)$ for $\al=1,\om,\ob$.  

 Suppose that $f,g\in B^\tm$ and $a:=f-g\in A^\tm$, so that
 $w(\pm a)=0$.  We then have $0=w(a)=w(f+(-g))\geq\min(w(f),w(g))$, so
 at least one of $w(f)$ and $w(g)$ must be nonpositive.  We also have
 $w(g)=w(f-a)\geq\min(w(f),0)$, which means that we cannot have
 $w(f)\geq 0>w(g)$, and similarly, we cannot have $w(g)\geq 0>w(f)$.
 Thus, if either of $w(f)$ or $w(g)$ is strictly negative, then both
 are.  If both are strictly negative, we can use the inequality
 $w(g)\geq\min(w(f),0)$ again to see that $w(g)\geq w(f)$.  Similarly,
 we can use the inequality $w(f)=w(g+a)\geq\min(w(g),0)$ to see that
 $w(f)\geq w(g)$, so $w(f)=w(g)$.

 Next note that $(1-\om)^2(1-\ob)^2(\om-\ob)^2=-27$, which shows that
 $\{1-\om,1-\ob,\om-\ob\}\subset A^\tm$.  Thus, the difference between
 any two of $\{\nu-1,\nu-\om,\nu-\ob\}$ lies in $A^\tm$, so we can
 apply the last paragraph.  This shows that at least two of
 $\{m_1,m_\om,m_\ob\}$ must be nonpositive, and if any of them are
 strictly negative then they are all equal.  In the latter case, the
 fact that $w\:B^\tm\xra{}\Z$ is surjective implies that
 $m_1=m_\om=m_\ob=-1$, so $w=v_\infty$.  If none of
 $\{m_1,m_\om,m_\ob\}$ is strictly negative then two of them must be
 zero and (by surjectivity) the third must be one, so $w$ is one of
 $\{v_1,v_\om,v_\ob\}$.

 It is clear that $\Aut(S)$ acts on $B=\CO_S$.  The action preserves
 the integral closure of $\Z[\third]$ in $B$, which is easily seen to
 be $A$, and it follows that $\Aut(S)$ acts on
 $V=\{v_\al\st\al\in\Om\}$, and thus on $\Om$.  More precisely, for
 any $\bt\in\Aut(S)$ we have an automorphism $\bt^*\:B\xra{}B$, and
 there is a unique permutation $\pi(\bt)\in\Perm(\Om)$ such that
 $v_\al(\bt^*(f))=v_{\pi(\bt)(\al)}(f)$ for all $\al\in\Om$ and
 $f\in\CO_S^\tm$.  This gives a homomorphism
 $\pi\:\Aut(S)\xra{}\Perm(\Om)$. 

 Next, we define automorphisms $\bt_0,\bt_1,\bt_2$ as shown in the
 table below.  It is straightforward to verify that the formulae given
 do indeed give automorphisms of $\CO_S$ and thus of $S$, and that the
 corresponding permutations are as listed.
 \[
  \renewcommand{\arraycolsep}{3em}
  \begin{array}{lll}
   \bt_0^*(\om)=\ob & \bt_0^*(\nu)=\nu
    & \pi(\bt_0)=(\om\;\ob) \\
   \bt_1^*(\om)=\om & \bt_1^*(\nu)=\om\nu
    & \pi(\bt_1)=(1\;\om\;\ob) \\
   \bt_2^*(\om)=\om & \bt_2^*(\nu)=(\nu+2)/(\nu-1) 
    & \pi(\bt_2)=(1\;\infty)(\om\;\ob)
  \end{array}
 \]
 It is not hard to check that these permutations generate
 $\Perm(\Om)$, so our map $\pi\:\Aut(S)\xra{}\Perm(\Om)$ is
 surjective. 

 Finally, suppose we have $\bt\in\Aut(S)$ with $\pi(\bt)=1$; we need
 to show that $\bt=1$.  It is clear that there must exist elements
 $u_1,u_\om,u_\ob\in A^\tm$ such that $\bt^*(\nu-\al)=u_\al(\nu-\al)$
 for all $\al\in\{1,\om,\ob\}$.  The case $\al=1$ gives
 $\bt^*(\nu)=u_1\nu+1-u_1$, and feeding this into the case $\al=\om$
 gives 
 \[ u_\om\nu-u_\om\om = 
   \bt^*(\nu-\om) = u_1\nu + 1 - u_1 - \bt^*(\om).
 \]
 On the other hand, $\bt^*(\om)$ is a primitive cube root of $1$ in
 $B$, and one checks that this gives $\bt^*(\om)\in\{\om,\ob\}$.  In
 particular, we have $\bt^*(\om)\in A$ so we can compare coefficients
 of $\nu$ to get $u_\om=u_1$ and 
 \[ \bt^*(\om) = 1 - u_1 + u_\om \om = 1 - u_1 + u_1\om . \]
 A similar argument gives $u_\ob=u_1$ and $\bt(\ob)=1-u_1+u_1\ob$.
 By multiplying these two equations together and simplifying we get  
 \[ 1=\bt^*(\om\ob)=(1-u_1+u_1\om)(1-u_1+u_1\ob)=1-3u_1+3u_1^2,  \]
 so $u_1=u_1^2$.  As $u_1\in A^\tm$ this gives $u_1=1$.  The above
 formulae then give $\bt^*(\nu)=\nu$ and $\bt^*(\om)=\om$,  so
 $\bt=1$ as required.
\end{proof}

\section{The level structure}

Define a function $\phi\:\F_3^2\xra{}\Map(S,\proj^2)$ as follows.
\begin{align*}
 \phi( 0, 0) &= [0:1:0]                                         \\
 \phi( 1, 0) &= [0:0:1]                                         \\
 \phi(-1, 0) &= [0:1-\nu^3:1]                                   \\
 \phi( 0, 1) &= [-(\nu-\ob)(\nu-\om):(\nu-\ob)^2(\nu-\om):1]    \\
 \phi( 1, 1) &= [-(\nu-  1)(\nu-\ob):(\nu-  1)^2(\nu-\ob):1]    \\
 \phi(-1, 1) &= [-(\nu-\om)(\nu-  1):(\nu-\om)^2(\nu-  1):1]    \\
 \phi( 0,-1) &= [-(\nu-\om)(\nu-\ob):(\nu-\om)^2(\nu-\ob):1]    \\
 \phi( 1,-1) &= [-(\nu-  1)(\nu-\om):(\nu-  1)^2(\nu-\om):1]    \\
 \phi(-1,-1) &= [-(\nu-\ob)(\nu-  1):(\nu-\ob)^2(\nu-  1):1]
\end{align*}
More compactly, when $l\neq 0$ we have 
\[ \phi(k,l) = [-(\nu-\om^a)  (\nu-\om^b):
                 (\nu-\om^a)^2(\nu-\om^b): 1],
\]
where $a=(k-1)l$ and $b=(k+1)l$.

\begin{proposition}
 The map $\phi$ actually lands in the group
 $\Gm(S,C)\subset\Map(S,\proj^2)$, and it is a homomorphism.
 Moreover, if $a\in\F_3^2\sm\{0\}$ then the locus where $\phi(a)=0$
 is the empty subscheme of $S$.
\end{proposition}
\begin{proof}
 We shall show how to reduce this to a direct calculation in $B$,
 which can be carried out by computer.  The amount of calculation
 required can be reduced by more careful arguments, but we leave the
 details to the reader.  First, we write $O=[0:1:0]$, which is the
 zero element for the usual group structure on $C$, and note that
 $\phi(0)=O$ as required.  Next, we note that for all
 $a\in\F_3^2\sm\{0\}$, the $z$-coordinate of $\phi(a)$ is
 invertible, so we can regard $\phi(a)$ as a section of the affine
 curve $C'=C\cap\aff^2\subset C$ via the usual correspondence
 $[x:y:z]\leftrightarrows(x/z,y/z)$.  Note that $C'$ is defined by the
 vanishing of the function
 \[ f(x,y) := y^2 + (\nu^3 - 1) y + 3 \nu x y - x^3. \]
 We next claim that $\phi(a)$ is actually a section of order three.
 By well-known arguments, it is enough to show that $\phi(a)$
 is an inflexion point, or equivalently that 
 \[ f(\phi(a) + t(1,\mu(a))) = 0 \pmod{t^3} \]
 where $\mu(a)$ is the slope of the curve at $\phi(a)$.  By standard
 formulae, the slope of the curve at a point $(x,y)$ is given by
 $3(\nu y-x^2)/(1-\nu^3-3\nu x-2y)$, and this gives the following
 values for $\mu(k,l)$.
 \begin{align*}
  \mu( 1, 0) &= 0                \\
  \mu(-1, 0) &= -3\nu            \\
  \mu( 0, 1) &= (\ob-1)(\nu-\ob) \\
  \mu( 1, 1) &= (\ob-1)(\nu-  1) \\
  \mu(-1, 1) &= (\ob-1)(\nu-\om) \\
  \mu( 0,-1) &= (\om-1)(\nu-\om) \\
  \mu( 1,-1) &= (\om-1)(\nu-  1) \\
  \mu(-1,-1) &= (\om-1)(\nu-\ob)
 \end{align*}
 More compactly, when $l\neq 0$ we have
 $\mu(k,l)=(\om^{-l}-1)(\nu-\om^{(k-1)l})$.

 We next claim that whenever $a\neq b$, the section $\phi(a)$ is
 nowhere equal to $\phi(b)$.  It is equivalent to say that the
 determinants of the $2\tm 2$ minors of the matrix $(\phi(a),\phi(b))$
 generate the unit ideal in $B$, which can be checked by direct
 calculation.

 Now suppose we have three distinct points $a,b,c\in\F_3^2$
 such that $a+b+c=0$.  Further direct calculations show that in each
 case the determinant of the $3\tm 3$ matrix
 $(\phi(a),\phi(b),\phi(c))$ is zero.  As the sections $\phi(a)$,
 $\phi(b)$ and $\phi(c)$ are everywhere distinct, it follows that
 $\phi(a)+\phi(b)+\phi(c)=0$.  Now consider instead the case where
 $a+b+c=0$ but $a$, $b$ and $c$ are not distinct.  If $a=b$ then
 $c=-2a=a$ so $a=b=c$, and similarly in all other cases.  Thus
 $\phi(a)+\phi(b)+\phi(c)=3\phi(a)=0$ by our earlier argument.  Thus
 $\phi(a)+\phi(b)+\phi(c)=0$ in all cases where $a+b+c=0$.

 By applying this to the cases of the form $(a,b,c)=(a,0,-a)$ we see
 that $\phi(-a)=-\phi(a)$ for all $a$.  Moreover, for all $a,b$ we
 have $a+b+(-a-b)=0$ so 
 \[ \phi(a)+\phi(b)-\phi(a+b) = \phi(a)+\phi(b)+\phi(-a-b) = 0,
 \]
 proving that $\phi$ is a homomorphism.
\end{proof}

\section{Automorphisms of $C$}

Let $\Aut(C,S)$ denote the group of pairs $(\al,\bt)$, where $\bt$ is
an automorphism of $S$ and $\al$ is an isomorphism $C\xra{}\bt^*C$ of
elliptic curves over $S$.  Equivalently, $\al$ is a map $C\xra{}C$
such that
\begin{itemize}
 \item[(a)] The following square is a pullback (and in particular is
  commutative): 
  \begin{diag}
   \node{C} \arrow{e,t}{\al} \arrow{s} \node{C} \arrow{s} \\
   \node{S} \arrow{e,b}{\bt}           \node{S}
  \end{diag}
 \item[(b)] For each point $s$ of $S$, the map
  $\al_s\:C_s\xra{}C_{\bt(s)}$ is a group homomorphism.
\end{itemize}

More concretely, the map $\bt$ corresponds to a ring automorphism
$\bt^*\:B\xra{}B$.  We write $\om'=\bt^*(\om)$, $\nu'=\bt^*(\nu)$,
$a'_k=\bt^*(a_k)$ and so on.  The map $\al$ extends canonically to an
automorphism of $S\tm\proj^2$ given by a matrix of the form 
\[ A = A(u,r,s,t) =
    \left(\begin{array}{ccc}
     u^2   & 0   & r    \\
     s u^2 & u^3 & t    \\
     0     & 0   & 1
    \end{array}\right),
\]
such that 
\begin{align*}
 u   a_1 &= a'_1 + 2 s \\
 u^2 a_2 &= a'_2 - s a'_1 + 3 r - s^2 \\
 u^3 a_3 &= a'_3 + r a'_1 + 2 t \\
 u^4 a_4 &= a'_4 - s a'_3 + 2 r a'_2 - s r a'_1 -
            t a'_1 + 3 r^2 - 2 s t \\
 u^6 a_6 &= a'_6 + r  a'_4 - t  a'_3 + r^2  a'_2 -
            t r a'_1 + r^3 - t^2.
\end{align*}
In fact, $\Aut(C,S)$ bijects with the set of pairs $(\bt,A)$ as above,
with composition given by
\[ (\bt_1,A_1)(\bt_0,A_0)=(\bt_1\bt_0,\bt_0^*(A_1)A_0). \]

\begin{proposition}
 There is a short exact sequence
 \[ \{\pm 1\} \xra{} \Aut(C,S) \xra{} \Aut(S). \]
\end{proposition}
\begin{proof}
 There is an evident homomorphism $\Aut(C,S)\xra{}\Aut(S)$, sending
 $(\bt,A)$ to $\bt$.  As $C$ is a group scheme over $S$ we have a map
 $-1\:C\xra{}C$ covering the identity map of $S$, which satisfies
 $(-1)^2=1$.  In terms of the description of $\Aut(C,S)$ given above,
 this is just the element $(1_S,A(-1,0,-a_1,-a_3))$.  This gives the
 first map in our sequence; it is clearly injective, and it is also
 clear that the composite $\{\pm 1\}\xra{}\Aut(S)$ is trivial.  We
 know from~\cite[where?]{de:cef} that away from the locus where
 $j\in\{0,1728\}$, elliptic curves have no automorphisms other than
 $\{\pm 1\}$.  The formula $j=27\nu^3(\nu^3+8)^3/(\nu^3-1)^3$ implies
 that the map $j\:S\xra{}\aff^1$ is dominant, and it follows that our
 sequence is exact in the middle.  To show that the right-hand map is
 surjective, we need only exhibit elements $(\bt_k,A_k)\in\Aut(C,S)$
 for $k=0,1,2$, where $\bt_k$ is as in the proof of
 Proposition~\ref{prop-Aut-S}.  The relevant matrices $A_k$ are as
 follows:
 \begin{align*}
  A_0 &= A(1,0,0,0) = I \\
  A_1 &= A(\om,0,0,0)   \\
  A_2 &= A\left(
           \frac{\ob-\om}{\nu-1} \;,\;
           3\frac{1-\nu^3}{(\nu-1)^3} \;,\;
           3\ob\frac{\nu-\om}{\nu-1} \;,\;
           3\frac{\nu^3-1}{(\nu-1)^4}((1-\om)+(1-\ob)\nu)
         \right).
 \end{align*}
\end{proof}

\section{The action of $GL_2(\F_3)$}

Recall the set $\Om$ and the isomorphism
$\pi\:\Aut(S)\xra{}\Perm(\Om)$ discussed earlier.  Define a bijection
$\xi\:\Om\xra{}P^1\F_3$ by
\begin{align*}
 \xi(1)      &= 0   \\
 \xi(\om)    &= 1   \\
 \xi(\ob)    &= -1  \\
 \xi(\infty) &= \infty,
\end{align*}
and let $\xi'$ denote the resulting isomorphism
$\Perm(\Om)\xra{}\Perm(P^1\F_3)$. 

\begin{proposition}
 There is a unique homomorphism $\gm\:\Aut(C,S)\xra{}GL_2(\F_3)$ such
 that for all $(\bt,\al)\in\Aut(C,S)$ and for all points $s$ of $S$,
 the following diagram commutes:
 \begin{diag}
  \node{\F_3^2} \arrow{s,l}{\phi_s} \arrow{e,t}{\gm(\bt,\al)}
  \node{\F_3^2} \arrow{s,r}{\phi_{\bt(s)}} \\
  \node{C_s} \arrow{e,b}{\al_s} \node{C_{\bt(s)}.}
 \end{diag}
 Moreover, this map $\gm$ is an isomorphism, and it makes the
 following diagram commute:
 \begin{diag}
  \node{\Aut(C,S)} \arrow[2]{e,t}{\gm} \arrow{s}
  \node[2]{GL_2(\F_3)} \arrow{s} \\
  \node{\Aut(S)} \arrow{e,b}{\pi}
  \node{\Perm(\Om)} \arrow{e,b}{\xi'}
  \node{\Perm(P^1\F_3).}
 \end{diag}
\end{proposition}
\begin{proof}
 Because $\phi$ is injective, there is at most one map $\gm(\bt,\al)$
 making the first diagram commute for all $s$.  If we can show that
 $\gm(\bt_k,A_k)$ exists for $k=0,1,2$ it will follow easily that
 $\gm(\bt,\al)$ exists for all $(\bt,\al)$ and moreover that $\gm$ is
 a homomorphism.  In fact, we have
 \begin{align*}
  \gm(\bt_0,A_0) &= \left(\begin{array}{cc}
                       1 & 0 \\ 0 & -1
                    \end{array}\right) \\
  \gm(\bt_1,A_1) &= \left(\begin{array}{cc}
                       1 & 1 \\ 0 & 1
                    \end{array}\right) \\
  \gm(\bt_2,A_2) &= \left(\begin{array}{cc}
                       0 & -1\\ 1 & 0
                    \end{array}\right).
 \end{align*}
 We next claim that the elements $(\bt_k,A_k)$ for $k=0,1,2$ generate
 $\Aut(C,S)$.  The elements $\bt_k$ certainly generate
 $\Aut(S)\simeq\Perm(\Om)$, so we see using our short exact sequence
 that it suffices to prove that $(1,-I)$ lies in the subgroup
 generated by the elements $(\bt_k,A_k)$.  However, $\bt_2^2=1$ so
 $(\bt_2,A_2)^2$ is either $(1,I)$ or $(1,-I)$, and the former is
 excluded by the fact that $\gm(\bt_2,A_2)^2=-I$.  This proves the
 claim, and in view of this we need only check that the second diagram
 commutes when evaluated at $(bt_k,A_k)$.  This can be done directly.
 For example, for $k=0$ we have
 $\pi(\bt_0)=(\om\;\;\ob)=(\xi^{-1}(1)\;\;\xi^{-1}(-1))$, so
 $\xi'\pi(\bt_0)=(1\;\;-1)$.  On the other hand,
 $\gm(\bt_0,A_0)=\bsm 1&0\\0&-1\esm$, so the associated M\"obius
 transformation of $P^1\F_3=\F_3\cup\{\infty\}=\{0,1,-1,\infty\}$ is
 the map $z\mapsto -z$, or equivalently the permutation $(1\;\;-1)$,
 as before.  The cases $k=1$ and $k=2$ are similar.  

 It is well-known that the map $GL_2(\F_3)\xra{}\Perm(P^1\F_3)$ is
 surjective, and that the kernel is the group of order $2$ generated
 by $-I$.  We have seen that $-I$ lies in the image of $\gm$ and that
 the map 
 \[ \Aut(C,S)\xra{}\Aut(S)\xra{}\Perm(\Om)\xra{}\Perm(P^1\F_3) \]
 is surjective, with kernel of order $2$.  It follows by diagram
 chasing that $\gm$ is an isomorphism.
\end{proof}

\section{Special fibres}

Let $S'\subset S$ be the locus where $\nu=0$, and put $C'=C\tm_SS'$.
This is given by the equation $y(1-y)+x^3=0$, and in $\CO_{S'}$ we
have
\begin{align*}
 a_1 &= a_2=a_4=a_6=c_4= 0 \\
 a_3 &= -1 \\
 c_6 &= -216 = -2^3 3^3 \\
 \Dl &= -27 = -3^3 \\
 j   &= 0.
\end{align*}
The curve has complex multiplication by $\Z[\om]$, given by the
formulae
\begin{align*}
 \om.(x,y) &= (\om x,y) \\
 \ob.(x,y) &= (\ob x,y) \\
 -(x,y)    &= (x,1-y).
\end{align*}
The points $\phi(k,l)$ are as follows:
\[ \begin{array}{|c|ccc|}
   \hline
      &     -1      &    0   &     +1      \\ \hline
   -1 & (-\ob,-\om) &  (0,1) & (-\om,-\ob) \\ 
    0 & (- 1 ,-\om) & \infty & (- 1 ,-\ob) \\ 
   +1 & (-\om,-\om) &  (0,1) & (-\ob,-\ob) \\ \hline
   \end{array}
\]

Now let $S''\subset S'$ be the locus where $2=\nu=0$, so
$\CO_{S''}=\F_2[\om]/(1+\om+\om^2)=\F_4$.  Put $C''=C\tm_SS''$.  

\begin{proposition}
 The curve $C''$ is supersingular, in other words the associated
 formal group has height $2$.
\end{proposition}
\begin{proof}
 We may work in the neighbourhood of $O$ where $y$ is invertible.  By
 putting $y=1$, we identify this with the affine scheme where
 $z-z^2=x^3$.  This gives $z\in\mxi_O^3$ and shows that $x$ is a
 formal parameter at $O$.  The equation $z-z^2=x^3$ gives
 \[ z=x^3+z^2=x^3+x^6+z^4=x^3+x^6+x^{12}+z^8=\ldots, \]
 and after completing at $\mxi_O$ we deduce that
 $z=\sum_{k\geq 0}x^{3.2^k}$.  Next, we recall the standard formula
 \[ -[x:y:z] = [x:-y-a_1x-a_3z:z]. \]
 In our context, this gives
 \[ -[x:1:z] = [x:z+1:z] = 
    \left[\frac{x}{1+z}:1:\frac{z}{1+z}\right].
 \]
 This means that $[-1](x)=x/(1+z)=x+x^4+O(x^5)$.  Moreover,
 $[2](x)=x-_F[-1](x)$ is a unit multiple of $x-[-1](x)=x^4+O(x^5)$,
 which proves that the height is $2$, as claimed.
\end{proof}
\begin{remark}
 One can in fact use standard duplication formulae and some
 rearrangement to show that
 \[ [2](x) = \frac{x^4}{1+z^4} = \sum_{k\geq 0}x^{12.2^k-8}. \]
\end{remark}

\section{Degeneration}

The curve $C$ can be extended in an obvious way over $\spec(\Z[\nu])$.

Over the locus where $\nu^3=1$, the curve is given by the equation
$y^2+3\nu xy=x^3$.  It is singular at the point $[0:0:1]$, and smooth
elsewhere.  If $3$ is invertible then the smooth locus is isomorphic
to $G_m$ by the map $u\mapsto[9u(u-1)\nu^2:27u:(u-1)^3]$.  The base is
the disjoint union of three pieces, where $\nu=1$, $\nu=\om$ and
$\nu=\ob$; for each of these pieces there is a subgroup $A<\F_3^2$
such that $\phi$ maps $A$ to the smooth locus by a homomorphism, and
carries the complement of $A$ to the singular point.  For example,
on the piece where $\nu=1$ we have $A=0\tm\F_3$.

On the other hand, over the locus where $\nu^3-1=3=0$ the curve is
just the cuspidal cubic $y^2=x^3$, and the smooth locus is isomorphic
to $G_a$ by the map $t\mapsto[t:1:t^3]$.  The map $\phi$ lands in (an
infinitesimal neighbourhood of) the singular locus.

\section{Landweber exactness}

\begin{proposition}
 The elliptic curve $C$ is Landweber exact.
\end{proposition}
\begin{proof}
 Because $C=C_0\tm_{S_0}S$ and the ring $B:=\CO_S$ is free of rank $2$
 over $B_0=\CO_{S_0}$, it suffices to prove that $C_0$ is Landweber
 exact.  Equivalently, for all primes $p$ we need to check that $p$ is
 not a zero-divisor in $B_0$, and that the Hasse invariant is not a
 zero-divisor in $B_0/p$.  When $p=3$ we have $B_0/p=0$, so everything
 is trivial. For other primes we have
 $B_0/p=\F_p[\nu][(1-\nu^3)^{-1}]$ which is an integral domain, so we
 need only show that the Hasse invariant is nontrivial.  We have
 $j=27\nu^3(\nu^3+8)^3/(\nu^3-1)^3$, which shows that the map
 $j\:S_0\xra{}\aff^1$ is nonconstant and thus dominant.  There are
 only finitely many supersingular $j$-invariants, so the Hasse
 invariant must be nontrivial as required.
\end{proof}

\begin{corollary}
 There is an essentially unique elliptic spectrum $E$ attached to $C$,
 and it has a compatible action of $G$.
\end{corollary}
\begin{proof}
 The category of Landweber exact elliptic spectra is equivalent to the
 category of Landweber exact elliptic curves.
\end{proof}

\section{The Weil pairing}

There is a pairing $e_n\:C[n]\tm C[n]\xra{}\mu_n$, defined as follows.
Given $P,Q\in C[n]$ we can find rational functions $g,h$ on $C$ such
that $\dv(g)=n[P]-n[O]$ and $\dv(h)=n[Q]-n[O]$.  After multiplying $g$
by a suitable scalar, we can assume that $g/h$ converges to $1$ at
$O$.  We define $e_n(P,Q)=(-1)^ng(Q)/h(P)$.  

\begin{proposition}
 $e_3(\phi(1,0),\phi(0,1))=\om$.
\end{proposition}
\begin{proof}
 Put 
 \begin{align*}
  P &= \phi(1,0) = [0:0:1] \\
  Q &= \phi(0,1) = [-(\nu-\ob)(\nu-\om):(\nu-\ob)^2(\nu-\om):1]. 
 \end{align*}
 Define
 \[ w = y + (1 - \ob)(\nu - \ob) x + (1 + \om)(\nu-\om)(\nu-\ob)^2 z \]
 and $g=y/z$, $h=w/z$.  As $x/y$ and $z/y$ converge to $0$ at
 $O=[0:1:0]$, we see that $h/g$ converges to $1$ at $O$.  I claim that
 $\dv(g)=3[P]-3[0]$ and $\dv(h)=3[Q]-3[0]$.  Assuming this, we have 
 \[ e_3(P,Q)= -\frac{g(Q)}{h(P)} =
    -\frac{(\nu-\ob)^2(\nu-\om)}{(1 + \om)(\nu-\om)(\nu-\ob)^2} =
    \frac{-1}{-\ob}=\om,
 \]
 as claimed.

 We now check that $\dv(g)=3[P]-3[O]$ and $\dv(h)=3[Q]-3[O]$.  First,
 it is well-known that $g$ has a pole of order $3$ at $O$, and $h$ is
 asymptotic to $g$ there so it also has a triple pole.  It is also
 clear that neither $g$ nor $h$ has any other poles.  In the finite
 part of the curve, we have $g=y$ and $f=y^2+(\nu^3-1)y+3\nu xy-x^3$.
 Thus the locus where $f=g=0$ is defined by the ideal $(y,x^3)$, which
 gives the point $P$ with multiplicity $3$.  Thus $\dv(g)=3[P]-3[O]$.

 Similarly, in the finite part of the curve we have
 $h=y+(1-\ob)(\nu-\ob)x+(1+\om)(\nu-\om)(\nu-\ob)^2$, so $h=0$ iff
 $y=-(1-\ob)(\nu-\ob)x-(1+\om)(\nu-\om)(\nu-\ob)^2$.  If we substitute
 this into $f$ we get $-(x+(\nu-\om)(\nu-\ob))^3$, which proves that
 $\dv(h)=3[Q]-3[O]$.
\end{proof}

\section{Universality}

Let $T$ be a base where $3$ is invertible, let $D$ be an elliptic
curve over $T$ with origin $O$, and let $\psi\:\F_3^2\xra{}\Gm(T,D)$
be a level three structure.  In other words, $\psi$ is a homomorphism
such that for all $a\in\F_3^2$ with $a\neq 0$, the locus where
$\psi(a)=0$ is empty.  

\begin{theorem}\label{thm-universal}
 There is a unique pullback square of the following type that carries
 $\psi$ to $\phi$:
 \begin{diag}
  \node{D} \arrow{e,t}{\al} \arrow{s} \node{C} \arrow{s} \\
  \node{T} \arrow{e,b}{\bt}           \node{S.}
 \end{diag}
 More precisely, the maps $\phi$ and $\psi$ are adjoint to maps
 $\phi^\#\:\F_3^2\tm S\xra{}C$ and $\psi^\#\:\F_3^2\tm T\xra{}D$, and
 the condition is that $\phi^\#\circ(1\tm\bt)=\al\circ\psi^\#$.
\end{theorem}

The rest of this section constitutes the proof.  

By a \emph{Weierstrass parametrisation} of $D$ we mean a pair of
functions $(x,y)$ on $D\sm\{O\}$ with poles of orders $2$ and $3$ at
$O$, such that $x^3/y^2$ tends to $1$ at $O$.  It is well-known that
such parametrisations exist locally on $T$, and are unique up to an
affine transformation of the form $x\mapsto u^2x+r$, 
$y\mapsto u^3y+su^2x+t$ with $u,r,s,t\in\CO_T$.  It is also well-known
that for any Weierstrass parametrisation, there are unique elements
$a_1,\ldots,a_6\in\CO_T$ such that
$y^2+a_1xy+a_3y=x^3+a_2x^2+a_4x+a_6$, and that the map
$(x,y)\:D\sm\{O\}\xra{}\aff^2\tm T$ gives an isomorphism of
$D\sm\{O\}$ with the locus where this equation is satisfied.

\begin{lemma}\label{lem-normal-param}
 Let $P$ be a section of $D[3]\sm\{O\}$.  Then locally on $T$ we can
 choose a Weierstrass parametrisation such that $x(P)=y(P)=0$ and
 $dy=0$ at $P$.  Moreover, these functions satisfy a unique equation
 of the form $y^2+a_1xy+a_3y=x^3$.
\end{lemma}
\begin{proof}
 Choose an arbitrary Weierstrass parametrisation.  After adding
 constants to $x$ and $y$ we may assume that $x=y=0$ at $P$.  Suppose
 that the corresponding Weierstrass equation is $f(x,y)=0$, where
 \[ f(x,y) = y^2 + a_1 x y + a_3 y - x^3 - a_2 x^2 - a_4 x - a_6. \]
 As $P$ lies on the curve this must be satisfied when $x=y=0$, so
 $a_6=0$.  As $P$ has order three we know that $(x(P),y(P))=(0,0)$ is
 an inflection point of the curve.  The function $f(0,t)=t^2$ does not
 vanish mod $t^3$ (at any geometric point) so the line $x=0$ is not
 the tangent line, so $dx$ generates the cotangent space at $P$.  This
 means that $dy=\al dx$ for some $\al\in\CO_T$.  After replacing $y$ by
 $y-\al x$ (and adjusting $a_i$ accordingly) we find that $dy=0$.
 This means that the tangent line is $y=0$, and $(0,0)$ is an
 inflection point so $f(0,t)=0$ mod $t^3$.  Thus $a_2=a_4=a_6=0$ and
 $f=y^2+a_1xy+a_3y-x^3$.
\end{proof}

\begin{lemma}\label{lem-non-parallel}
 Let $P,Q$ be sections of $D[3]\sm\{O\}$ that are everywhere linearly
 independent over $\F_3$, and let $(x,y)$ be a Weierstrass
 parametrisation.  Then $(dy/dx)_Q-(dy/dx)_P$ is invertible.
\end{lemma}
\begin{proof}
 As any two parametrisations are related by an affine transformation,
 we may replace the given parametrisation by any other one without
 changing the statement.  Thus, by the previous lemma, we may assume
 that $x(P)=y(P)=(dy/dx)_P=0$, and that we have an equation of the
 form $f(x,y)=y^2+a_1xy+a_3y-x^3=0$.  We now need to show that
 $(dy/dx)_Q$ is invertible.  We will identify $D\sm\{O\}$ with its
 image under the map $(x,y)\:D\sm\{O\}\xra{}\aff^2\tm T$; let $(c,d)$
 be the point corresponding to $Q$.  It is standard that the line
 $x=0$ meets $D\sm\{O\}$ only at $\pm P$, and $Q$ is everywhere
 linearly independent of $P$ so $c$ is invertible.  Next, note that
 the coefficient of $t^2$ in $f(c+t,d)$ is $-3c$, so
 $f(c+t,d)\neq 0\pmod{t^3}$.  As $Q$ is an inflection point, this
 means that the tangent line at $Q$ cannot be horizontal, so
 $(dy/dx)_Q\neq 0$.  This holds at every geometric point, so
 $(dy/dx)_Q$ is invertible as claimed.
\end{proof}

\begin{proof}[Proof of Theorem~\ref{thm-universal}]
 We'll write $P=\psi(1,0)$ and $Q=\psi(0,1)$ and $\om=e_3(P,Q)$.  This
 satisfies $1+\om+\om^2=0$, by standard properties of the Weil
 pairing.  

 Now choose a Weierstrass parametrisation $(x,y)$ such that
 $x(P)=y(P)=(dy/dx)_P=0$.  For any $R\in D\sm\{O\}$ we write
 $\mu(R)=(dy/dx)_R$, so $\mu(P)=0$.  If $R\neq\pm P$,
 Lemma~\ref{lem-non-parallel} tells us that $\mu(R)$ is invertible.
 We may thus define $\lm=\mu(Q+P)^{-1}-\mu(Q-P)^{-1}$.  If we replace
 $x$ by $u^2x$ and $y$ by $u^3y$ then $\mu(R)$ becomes $u\mu(R)$ for
 all $R$, so $\lm$ becomes $\lm/u$.  By taking $u=\lm/3$ and
 performing this replacement, we may assume that $\lm=3$.  It is not
 hard to check that the resulting parametrisation $(x,y)$ is uniquely
 specified by these constraints.  We will identify $D\sm\{O\}$ with
 its image under the map $(x,y)\:D\sm\{O\}\xra{}\aff^2\tm T$, which as
 usual is defined by an equation $f(x,y)=0$ where
 $f(x,y)=y^2+a_1xy+a_3y-x^3$.  We define $\nu=a_1/3$.

 Remaining details are left to the reader.
\end{proof}

\end{document}